\documentclass[12pt]{amsart}

\usepackage{amsmath}
\usepackage{amsfonts}
\usepackage{amssymb}

\usepackage{amsthm}

\usepackage{graphicx, color}
\usepackage{nicefrac}
\usepackage[utf8]{inputenc}

\usepackage{amscd}
\usepackage{amssymb}
\usepackage{latexsym}
\usepackage{amsthm}

\usepackage[colorlinks=true,linkcolor=blue,citecolor=blue]{hyperref}
\usepackage{fourier}

\usepackage{hyperref}
\DeclareMathOperator{\mon}{mon}
\DeclareMathOperator{\ct}{cot}
\DeclareMathOperator{\ctp}{cot_{Hyp}}
\DeclareMathOperator{\fct}{cot_{\mathfrak F}}

\DeclareMathOperator{\Po}{\mathcal P}
\DeclareMathOperator{\id}{id}

\DeclareMathOperator{\mult}{mult}

\newtheorem{theorem}{Theorem}[section]
\newtheorem{lemma}[theorem]{Lemma}
\newtheorem{corollary}[theorem]{Corollary}
\newtheorem{proposition}[theorem]{Proposition}
\theoremstyle{definition}
\newtheorem{definition}[theorem]{Definition}

\theoremstyle{remark}
\newtheorem{remark}[theorem]{Remark}

\begin{document}

\title{Some polynomial versions of cotype and applications}

%\author{Daniel Carando\fnref{dani}}
%\ead{dgcarando@dm.uba.ar}
%\author{Daniel Galicer\corref{cor1}\fnref{dany}}
%\ead{dgalicer@dm.uba.ar}
%\fntext[dani]{Partially supported by ANPCyT PICT 05 17-33042, UBACyT Grant X038 and ANPCyT PICT 06 00587.}
%\fntext[dany]{Partially supported by ANPCyT PICT 05 17-33042, UBACyT Grant X863 and a CONICET doctoral fellowship.}
%
%\address{Departamento de Matem\'{a}tica, Facultad de Ciencias Exactas y Naturales, Universidad de Buenos Aires,\\ Pab. I, Cdad Universitaria (1428)  Buenos Aires, Argentina and CONICET.}
%

\author{Daniel Carando}
\author{Andreas Defant}
\author{Pablo Sevilla-Peris}

\thanks{The first author was partially supported by CONICET-PIP 0624,
UBACyT 20020130100474BA and  ANPCyT PICT 2011-1456. The second and third
authors were supported by MICINN  MTM2011-22417 and the third author also by UPV-SP20120700}

\address{Daniel Carando. Departamento de Matem\'{a}tica - Pab I,
Facultad de Cs. Exactas y Naturales, Universidad de Buenos Aires,
(1428) Buenos Aires, Argentina and IMAS-CONICET} \email{dcarando@dm.uba.ar}

\address{Andreas Defant. Institut f\"ur Mathematik. Universit\"at Oldenburg. D-26111 Oldenburg, Germany.} \email{defant@mathematik.uni-oldenburg.de}

\address{Pablo Sevilla-Peris. Instituto Universitario de Matem\'atica Pura y Aplicada. Universitat Polit\`ecnica de Val\`encia. 46022 Valencia, Spain.} \email{psevilla@mat.upv.es}



\begin{abstract}
We introduce non-linear versions of the classical cotype of Banach spaces. We show that spaces with l.u.st and cotype, and that spaces having Fourier cotype enjoy our non-linear cotype. We apply these concepts to get results on convergence of vector-valued power series in infinite many variables and on $\ell_{1}$-multipliers of vector-valued Dirichlet series. Finally we introduce cotype with respect to indexing sets, an idea that includes our previous definitions.
\end{abstract}

\maketitle

%\address{Departamento de Matem\'{a}tica, Facultad de Ciencias Exactas y Naturales, Universidad de Buenos Aires,\\ Pab. I, Cdad Universitaria (1428)  Buenos Aires, Argentina and CONICET.} \email{dcarando@dm.uba.ar} \email{dpinasco@dm.uba.ar}
%\email{dgalicer@dm.uba.ar}

%\tableofcontents

%\textbf{Theorem.} \emph{(Bombal--P\'erez Garc\'ia--Villanueva)
% $T: c_{0} \times \cdots \times c_{0} \to X$ ($X$ with cotype $q$)
%\[
% \bigg( \sum \Vert a_{i_{1}, \ldots , i_{m}} \Vert^{q} \bigg)^{1/q} \leq C_{q}(X)^{m} \Vert T \Vert \, .
%\]
%}
%
%
%\bigskip
%
%
%\textbf{Theorem.} \emph{(Defant--Popa--Schwarting)
%There exists $C>0$ such that
% $T: c_{0} \times \cdots \times c_{0} \to X$, $v:X \to Y$ $(r,1)$--summing ($Y$ with cotype $q$)\\
%$\exists \rho = \rho(m,r,q)$
%\[
% \bigg( \sum \Vert v a_{i_{1}, \ldots , i_{m}} \Vert_{Y}^{\rho} \bigg)^{1/\rho} \leq C^{m} \Vert T : c_{0} \times \cdots \times c_{0} \to X \Vert \, .
%\]
%}
%

\section{Homogeneous cotype}

Cotype,  introduced in the 1970's by Maurey and Pisier, is one of the cornerstones of the modern Banach space theory. We recall that a complex Banach space $X$ has cotype $q$ if there is a constant $C>0$ such that for any choice of finitely many vectors
$x_1, \ldots, x_N \in X$ we have
\begin{equation}\label{cotype}
\Big(\sum_{k=1}^N \big\| x_k \big\|^q\Big)^{1/q}  \leq C \Big( \int_{\mathbb{T}^N} \Big\| \sum_{k=1}^N  x_k z_k  \Big\|^{2} dz \Big)^{1/2}\,.
\end{equation}
Here  $\mathbb{T}^N$ is the  $N$-dimensional torus (the $N$-th product of $\mathbb{T}:= \big\{z \in \mathbb{C} : |z|=1\big\}$) endowed with the $N$th product of the normalized Lebesgue measure on $\mathbb{T}$. We will later use the same notation for $N= \infty$.\\

Cotype is a property of the Banach space $X$ in terms of linear mappings in the variables $z_1,z_2,\dots$ with values in $X$. Our aim in this note is to consider cotype-like properties which consider not only linear mappings,
but also other algebraic combinations: polynomials (of certain classes) in the variables $z_1,z_2,\dots$ with values in $X$. For this, we introduce the following notation: if $\alpha\in \mathbb{N}_{0}^{(\mathbb{N})}$ is a multi
index (a finite sequence on $\mathbb{N}_{0}$ of arbitrary length), we write  $z^{\alpha}$ for the monomial $z_{1}^{\alpha_{1}}\cdots z_n^{\alpha_n}$, and set $|\alpha|:=\alpha_1+\alpha_2+\cdots$ .\\
For each $m$-homogeneous polynomial on $N$ variables
\[
P(z)= \sum_{\substack{\alpha \in \mathbb{N}_0^N \\ |\alpha|=m}} x_\alpha z^\alpha
\]
there exists a unique symmetric $m$-linear form $ T(z^{(1)}, \ldots,z^{(m)}) = \sum_{i_1, \ldots,i_m=1}^N a_{i_1, \ldots,i_m} z^{(1)}_{i_1} \ldots z^{(m)}_{i_m}$. Then \cite[Proposition~2.1]{CarandoDefantSevilla2013}
and the relation between the coefficients $x_{\alpha}$ and $a_{i_1, \ldots,i_m}$ (see e.g. \cite[page~544]{DeGaMaPG08} or \cite[Lemma~2.5]{BayartDefantFrerickMaestreSevilla2014}) immediately give
that for every finite family $(x_{\alpha} )_{\vert \alpha \vert =m}$ (i.e., a family with only finitely many non-zero elements) we have
%
%We begin by stating the following result from \cite{CarandoDefantSevilla2013}, in which $K$ denotes the constant in the $(2,1)$-Kahane inequality.
%
%
%\begin{proposition} \label{maintool}
%Let $X$ be a Banach space of cotype $q$, $2 \leq q < \infty$, and
%\[
%P: \mathbb{C}^N \rightarrow X,\,\,\, P(z)= \sum_{\substack{\alpha \in \mathbb{N}_0^N \\ |\alpha|=m}} x_\alpha z^\alpha
%\]
%be an $m$-homogeneous polynomial. Let
%\[
%T: \mathbb{C}^N \times \ldots \times  \mathbb{C}^N \rightarrow X,\,\,\,
% T(z^{(1)}, \ldots,z^{(m)}) = \sum_{i_1, \ldots,i_m=1}^N a_{i_1, \ldots,i_m} z^{(1)}_{i_1} \ldots z^{(m)}_{i_m}
%\]
%be the unique $m$-linear symmetrization of $P$. Then
%$$
%\Big(
%\sum_{i_1,\ldots,i_m}
%\big\| a_{i_1, \ldots,i_m} \big\|_X^q
%\Big)^{1/q}
%\leq
%\big(C_q(X)\, K \big)^{m} \frac{m^m}{m!}
%\,\, \int_{\mathbb{T}^N} \big\| P(z) \big\|_X dz\,.
%$$
%\end{proposition}
%
%
%
%The previous result and  the relationship EXPLAIN THE RELATIONSHIP??? between the $(x_\alpha)_{\alpha}$ and $(a_{i_1, \ldots,i_m})_{i_1, \ldots,i_m}$ give:
\begin{equation}\label{homog-cotype}
\bigg( \sum_{\vert \alpha \vert = m} \Vert x_{\alpha} \Vert^{q} \bigg)^{1/q}  \leq  \big(C_q(X)\, K \big)^{m} \frac{m^m}{m!} (m!)^{1/q'} \bigg(\int_{\mathbb{T}^{\infty}} \Big\Vert  \sum_{\vert \alpha \vert = m}  x_{\alpha} z^{\alpha} \Big\Vert^2 dz\bigg)^{1/2} \,,
\end{equation}
where $C_{q}(X)$ denotes the best constant in \eqref{cotype}, $K$ denotes the constant in the $(2,1)$-Kahane inequality (see e.g. \cite[Theorem~11.1]{DiestelJarchowTonge}) and $\frac{1}{q}+ \frac{1}{q'}=1$. Let us observe that in the right-hand side we are actually integrating on some finite dimensional $\mathbb{T}^{N}$, where $N$ is given by the longest $\alpha$ involved in the $x_{\alpha}$'s.\\
Note that letting $m=1$ in this inequality we have exactly \eqref{cotype}. Hence, \eqref{homog-cotype} can be seen as a sort of homogeneous version of the classical cotype. We will then say that $X$ has
\emph{$m$-homogeneous cotype} $q$ if there exists a constant $C>0$ such that for any finite multi indexed sequence $(x_\alpha)_{|\alpha|=m}\subset X$ we have
\begin{equation}\label{def-homog-cotype}
\bigg( \sum_{\vert \alpha \vert = m} \Vert x_{\alpha} \Vert^{q} \bigg)^{1/q}  \leq  C\, \bigg(\int_{\mathbb{T}^{\infty}} \Big\Vert  \sum_{\vert \alpha \vert = m}  x_{\alpha} z^{\alpha} \Big\Vert^2 dz\bigg)^{1/2} \,.
\end{equation}
The \emph{constant of $m$-homogeneous cotype}, which we denote by $C_{q,m}(X)$, will be the best constant for which the inequality holds.

With this definition, what \eqref{homog-cotype} is telling us is that if  $X$ has cotype $q$, then it also has $m$-homogeneous cotype $q$ with $ C_{q,m}(X)\le \big(C_{q}(X)\, K \big)^{m} \frac{m^m}{m!} (m!)^{1/q'}$. On the other hand, it is easy to see that if $X$ has $m$-homogeneous cotype $q$ for some $m$ and $q$,  then $X$ has cotype $q$ with $C_{q}(X)\le C_{q,m}(X)$.

In other words, cotype and $m$-homogeneous cotype  are equivalent properties. This fact has interesting consequences for vector-valued power and Dirichlet series (see e.g. \cite{CarandoDefantSevilla2013}), but for some applications (see Section~\ref{sec-applications}) a better control of the behaviour of $C_{q,m}(X)$ as $m$ grows is needed. When we do have such control, we say that the Banach space $X$ has \emph{hypercontractive homogeneous cotype.}

\begin{definition} \label{def:polcot}
A Banach space $X$ has hypercontractive homogeneous cotype $q$ if there exists $C >0$ such that for every $m\in \mathbb N$ and every finite family  $(x_{\alpha} )_{\vert \alpha \vert =m}$ we have
\[
 \bigg( \sum_{\vert \alpha \vert = m} \Vert x_{\alpha} \Vert^{q} \bigg)^{1/q} \leq C^{m} \bigg( \int_{\mathbb{T}^{\infty}} \Big\Vert  \sum_{\vert \alpha \vert = m}  x_{\alpha} z^{\alpha} \Big\Vert^{2} dz \bigg)^{1/2} \,.
\]
\end{definition}

Hypercontractive homogeneous cotype is clearly a local property, and it means $m$-homo\-geneous cotype for all $m$ together with an estimate of the form $C_{q,m}(X)\le C^m$ for some universal constant $C$. We consider
\[
\ct(X) := \inf \Big\{2 \le q< \infty \,\big\vert \, X \text{ has cotype } q \Big\}\,
\]
and
\[
\ctp (X)  := \inf \Big\{2 \le q< \infty \,\big\vert \, X \text{ has hypercontractive homogeneous cotype } q \Big\}\,.
\]
Although these infimums are in general not attained we call them the optimal cotype and the optimal hypercontractive homogeneous cotype of $X$. If there is no $2\le q <\infty $ for which $X$ has (hypercontractive homogeneous) cotype $q$, then $X$ is said to have trivial (hypercontractive homogeneous) cotype, and we put $\ct(X)= \infty$ (or $\ctp(X) = \infty$).\\
Clearly, if $X$ has hypercontractive homogeneous cotype, then it has (classical) cotype or, in other words, $\ct(X) \leq \ctp (X)$ for every Banach space $X$.
We conjecture that these two concepts are actually equivalent; that is: a Banach space has hypercontractive homogeneous cotype $q$ if and only if it has cotype $q$.\\
We are not able to prove our conjecture, but we give some positive answers. First we show that for spaces having local unconditional structure it is true (Theorem~\ref{lust}). We  prove that spaces having Fourier cotype also
have hypercontractive homogeneous cotype (Proposition~\ref{fourier-implies-pol}). As a consequence we have that  for Schatten classes $\mathcal{S}_{r}$ with $r\ge 2$
our conjecture is  true, and also that for Banach spaces with type 2 the equality $\ct(X) = \ctp (X)$ holds.

By Kahane's inequality (see e.g. \cite[Theorem~11.1]{DiestelJarchowTonge}), the $L_2$ norm at the right-hand side of the inequality in \eqref{cotype} can be changed by any other $L_p$-norm. Before we go into details, we give a kind of  polynomial version of  Kahane's inequality. This shows that in Defintion~\ref{def:polcot} we can take any $L_p$-norm at the right hand side, just as in the usual definition of cotype. A recent result \cite[Theorem~2.1]{Schoenberg} shows that the constant $(  \nicefrac{r}{s} )^{m/2}$ is almost optimal in this case.

\begin{proposition}\label{pol-kahane}
For  $1\le s \le r<\infty$, any Banach space $X$ and any finite sequence $(x_{\alpha})_{\vert \alpha \vert = m} \subset X$ we have
\[
\bigg( \int_{\mathbb{T}^{N}} \Big\Vert \sum_{\vert \alpha \vert = m} x_{\alpha} z^{\alpha}\Big\Vert^{r}
  dz
 \bigg)^{1/r} \leq \left(\frac{r}{s} \right)^{\frac{m}{2}}
\bigg( \int_{\mathbb{T}^{N}} \Big\Vert \sum_{\vert \alpha \vert = m} x_{\alpha} z^{\alpha}\Big\Vert^{s}  dz \bigg)^{1/s} \,.
\]
\end{proposition}

For the proof of  Proposition~\ref{pol-kahane}, we introduce vector-valued Hardy spaces. We define them in a more general setting than needed for this proof, since we will come back to them later in Section~\ref{sec-applications}.
%We denote by $dw$ the normalized Lebesgue measure on the  infinite dimensional polytorus $\mathbb{T}^{\infty} = \prod_{k=1}^\infty \mathbb{T}$, e.g. the countable product measure  of the normalized Lebesgue measure
%on $\mathbb{T}$.
For any  multi index $\alpha = (\alpha_1, \dots, \alpha_n,0, \ldots ) \in \mathbb{Z}^{(\mathbb{N})}$ (all finite sequences in $\mathbb{Z}$) the $\alpha$th Fourier coefficient $\hat{f}(\alpha)$ of  $f  \in L_{1} (\mathbb{T}^{\infty}, X)$ is given by
\[
\hat{f}(\alpha) = \int_{\mathbb{T}^{\infty}} f(z) z^{- \alpha} dz\,.
\]
Then, given $1 \le r \le \infty$, the $X$-valued Hardy space on $\mathbb{T}^\infty$ is the subspace of $L_{r} (\mathbb{T}^{\infty}, X)$ defined as
\[
H_{r}(\mathbb{T}^{\infty},X) = \Big\{ f \in  L_{r} (\mathbb{T}^{\infty},X) \,\,\big|\,\, \hat{f}(\alpha) = 0 \,  , \, \,\, \,
\forall \alpha \in \mathbb{Z}^{(\mathbb{N})} \setminus \mathbb{N}_{0}^{(\mathbb{N})} \Big\} \,.
\]
The spaces $H_{r}(\mathbb{T}^N,X)$ with $N\in \mathbb N$, are defined analogously.

Given $f\in H_{s}(\mathbb{T},X)$ and $0 < c <1$, we  define for $z=e^{i\theta}$ the Poisson integral $$\Po_c(f)(z):=\frac 1 {2\pi} P_c\ast f( z)= \frac 1 {2\pi}\, \int_{-\pi}^{\pi} f(e^{it}) P_c(\theta -t) dt, $$where $P_c$ denotes the Poisson kernel
$$P_c(t)= \sum_{n=-\infty}^{\infty}c^{|n|}e^{int} =\frac{1-c^2}{1-2c\cos(t)+c^2} .$$
Equivalently, $\Po_c(f)$ can be defined as the function whose Fourier coefficients are
\[
\widehat{\Po_c(f)}(n)=c^n \hat{f}(n),\quad \text{for } n\in \mathbb N_0 .
\]

As in the scalar valued case, the Poisson integral gives an `extension' of $f\in H_{s}(\mathbb{T},X)$ to a function $F$ on the disc $\mathbb D$, defining for $w=\rho e^{i\theta}\in \mathbb D$:
\begin{equation}\label{poison-extension}F(w) = \Po_{\rho}(f)(e^{i\theta}) = \sum_{n=0}^{\infty} \hat f(n) w^n.\end{equation}
 For $s=+\infty$, we also have $$\sup_{w\in \mathbb D}|F(w)| = \|f\|_{H_\infty ( \mathbb{T},X)}.$$ We refer to \cite{Bla88} and the references therein for details.
Just for completeness, we comment that going the other way around (i.e., starting with a function on the disc and taking its boundary values to get a function on the torus) is not always possible in the vector-valued case.
This is true if and only if $X$ has the analytic Radon-Nikod\'ym property.

The operator $\Po_{c} : H_{s}(\mathbb{T}^N,X) \to H_{s}(\mathbb{T}^N,X)$ is a linear contraction, since it is given by the convolution with a function of $L_1$-norm one (note the normalization by $2\pi$). Weissler in \cite{We80} proved that if $r>s\ge 1$, then $\Po_{c} :H_{s}(\mathbb{T},\mathbb{C})  \to
H_{r}(\mathbb{T},\mathbb{C})$ is again a contraction for every $c \le \sqrt{\nicefrac s r}$ and this value is optimal. We use now his result to give a vector-valued version.

\begin{lemma} \label{vector-valued-weissler}
Let $r>s\ge 1$ and set $c=\sqrt{\nicefrac s r}$. Then, the mapping $\Po_c$ is a linear contraction from $H_{s}(\mathbb{T},X)$ to $H_{r}(\mathbb{T},X)$.
\end{lemma}
\begin{proof}
Take  $f\in H_{s}(\mathbb{T},X)$ and $\varepsilon >0$. By classical results (as in  \cite[Theorem~ 2.7]{HaPi93}) we can find $\varphi\in H_s$ (the scalar-valued space) with $|\varphi(z)|=\|f(z)\|_X + \varepsilon$ for all $z\in \mathbb T$ and $g\in H_\infty ( \mathbb{T},X)$ with $\|g\|_{H_\infty ( \mathbb{T},X)}\le 1$ such that $f=\varphi \, g$. Now, if we call $F$, $\varPhi$ and $G$ the extensions of $f$, $\varphi$ and $g$ given by \eqref{poison-extension}, we have
\begin{multline*}
\|\Po_c f\|_{H_{r}(\mathbb{T},X)}  =  \Big( \int_{\mathbb T} \| F(cz) \|^r_X dz \Big)^{1/r}= \Big( \int_{\mathbb T} \| \varPhi(c z) G(c z) \|^r_X dz \Big)^{1/r}  \\
\le  \|g\|_{H_\infty ( \mathbb{T},X)}  \Big( \int_{\mathbb T} | \varPhi( c z)  |^r dz \Big)^{1/r} \\ = \|g\|_{H_\infty ( \mathbb{T},X)}  \Big( \int_{\mathbb T} | \Po_c(\varphi)(z)  |^r dz \Big)^{1/r}
\le \Big( \int_{\mathbb T} | \varphi(z)  |^s dz \Big)^{1/s},
\end{multline*}
where the last inequality is a consequence of \cite[Corollary~2.1]{We80}. Now, the last expression is not greater than
\[
\Big( \int_{\mathbb T} \| f(z)+\varepsilon  \|^s_X dz \Big)^{1/s} \le \Big( \int_{\mathbb T} \| f(z)  \|^s_X dz \Big)^{1/s} + \varepsilon = \| f\|_{H_{s}(\mathbb{T},X)} + \varepsilon \,.
\]
Since this holds for any $\varepsilon >0$, the proof is complete.
%CAN WE ACTUALLY EVALUATE THESE $\phi$ AND $G$ IN THE TORUS? I THINK HAGERUP AND PISIER DO SO, BUT MAYBE WE SHOULD DO THIS MORE CAREFULLY... OR AT LEAST SAY SOMETHING?
\end{proof}

Note that if we take $f$ in the lemma to be a polynomial, we can rephrase the result as
\begin{equation}\label{poisson}
\Big( \int_{\mathbb T} \left\| \sum_{k=0}^N x_k \left(cz \right)^k \right\|^r_X dz \Big)^{1/r}  \le \Big( \int_{\mathbb T} \left\| \sum_{k=0}^N x_k z ^k \right\|^s_X dz \Big)^{1/s}
\end{equation}

Iterating as in \cite[Theorem~9]{Ba02}, working with one variable at a time and applying the continuous
Minkowski inequality, we can deduce from \eqref{poisson} that $\Po_c$ is also a continuous contraction from $H_{s}(\mathbb{T}^N,X)$ to $H_{r}(\mathbb{T}^N,X)$. For $m$-homogeneous polynomials this gives:
$$ \bigg( \int_{\mathbb{T}^{N}} \Big\Vert \sum_{\vert \alpha \vert = m} x_{\alpha} \left(cz \right)^{\alpha}\Big\Vert^{r} \bigg)^{1/r} \leq
\bigg( \int_{\mathbb{T}^{N}} \Big\Vert \sum_{\vert \alpha \vert = m} x_{\alpha} z^{\alpha}\Big\Vert^{s} \bigg)^{1/s},$$
which by the homogeneity of the polynomial yields Proposition~\ref{pol-kahane}.

\section{Banach spaces with hypercontractive homogeneous cotype}

For $q \ge 2$, Banach lattices with nontrivial concavity $q $ have hypercontractive homogeneous cotype
$q$. This fact  can be deduced from an analysis of the proof of  \cite[Theorem~5.3]{DeMaSch12}; use in a first step  Krivine's calculus  to extend
\cite[Theorem~9]{Ba02} to Banach lattices and then in a second step the concavity property of the Banach lattice. In this section we give other Banach spaces, different from lattices, that have hypercontractive homogeneous cotype.

\subsection{Local unconditional structure and hypercontractive homogeneous cotype}\label{subsec-lust}

Our next result shows that every Banach space with  local unconditional structure (l.u.st.) and cotype $q$ has hypercontractive homogeneous cotype $q$, giving the first positive answer to our conjecture. Let us recall (see e.g. \cite[Definition~1.1]{Pi78} or \cite[Chapter~17]{DiestelJarchowTonge}) that a Banach space $X$ is said to have local unconditional structure if there exists $\lambda >0$ such that for every finite dimensional subspace $F$ of
$X$ there exists a space $U$ with unconditional basis $\{u_{n}\}$ and operators $T:F \to U$ and $S:U \to F$ such that $ST=\id_{F}$ and $\Vert T \Vert \cdot \Vert S \Vert \cdot \chi\{u_{n} \} \leq \lambda$.

\begin{theorem} \label{lust}
If  $X$ has cotype $q$ and l.u.st.,  then $X$ has hypercontractive homogeneous cotype $q$.
\end{theorem}

The theorem will be a direct consequence of the next two results.
Pisier in \cite{Pi78} introduced what is now usually called \emph{Pisier's property $(\alpha)$}. The next simple lemma  shows that if $X$ has cotype $q$ and satisfies \eqref{estreshita}, which is a polynomial weaker version
of property $(\alpha)$ with good constants, then $X$ has hypercontractive homogeneous cotype $q$. Then Proposition~\ref{rubio}  shows that if $X$ has cotype $q$ and l.u.st., then it satisfies a strong version of property \eqref{estreshita}.

\begin{lemma}
 Let $X$ be a Banach space with cotype $q$ and suppose there exists $C>$ such that for every finite family $(x_{\alpha})_{\alpha \in \mathbb{N}_{0}^{N} ,\, \vert \alpha \vert =m} \subset X$,
\begin{equation} \label{estreshita} \tag{$\bigstar$}
\bigg( \int_{\Omega} \int_{\mathbb{T}^{N} } \Big\Vert  \sum_{\substack{\alpha \in \mathbb{N}_{0}^{N} \\ \vert \alpha \vert = m}}  x_{\alpha} \varepsilon_{\alpha} (\omega) z^{\alpha} \Big\Vert^{2} dz d \omega \Bigg)^{1/2}
\!\!\!\!\!\!\!\! \leq C^{m} \bigg( \int_{\mathbb{T}^{N} } \Big\Vert  \sum_{\substack{\alpha \in \mathbb{N}_{0}^{N} \\ \vert \alpha \vert = m}}  x_{\alpha}  z^{\alpha} \Big\Vert^{2} dz  \Bigg)^{1/2} \,.
\end{equation}
where $(\varepsilon_{\alpha})$ are i.i.d. Rademacher random variables. \\
Then $X$ has hypercontractive homogeneous cotype $q$.
\end{lemma}
 \begin{proof}
Let $C_q$ be the  cotype $q$ constant of $X$. For each $z\in \mathbb{T}^{N}$, since $(x_{\alpha})$ is a finite family we have
%\begin{eqnarray*}
% \bigg(  \sum_\alpha  \Vert x_{\alpha} \Vert^{q} \bigg)^{1/q} & =& \bigg(  \sum_\alpha  \Vert x_{\alpha} z^\alpha \Vert^{q} \bigg)^{1/q} \le  C_q  \bigg( \int_{\Omega} \Big\Vert \sum_\alpha \varepsilon_{k} (\omega) x_{\alpha} z^\alpha \Big\Vert^{2} d \omega \bigg)^{1/2} \\ &\le & C_q (\sqrt 2)^m \, \int_{\Omega} \Big\Vert \sum_\alpha \varepsilon_{k} (\omega) x_{\alpha} z^\alpha \Big\Vert d \omega ,
%\end{eqnarray*}where the polynomial Kahane  inequality (Proposition~\ref{pol-kahane})  was used in the last step.
%Integrating this inequality on $z\in \mathbb{T}^{N}$ we obtain
%\begin{eqnarray*}
% \bigg(  \sum_\alpha  \Vert x_{\alpha} \Vert^{q} \bigg)^{1/q} &\le &   C_q (\sqrt 2)^m  \int_{\mathbb{T}^{N} }   \int_{\Omega} \Big\Vert \sum_\alpha \varepsilon_{k} (\omega) x_{\alpha} z^\alpha \Big\Vert d \omega  \, dz  \\
% &\le &  C_q (\sqrt 2)^m   \bigg(  \int_{\Omega} \int_{\mathbb{T}^{N} } \Big\Vert \sum_\alpha \varepsilon_{k} (\omega) x_{\alpha} z^\alpha \Big\Vert^2  dz  \, d \omega  \Bigg)^{1/2} \\ &\le &  C_q (\sqrt 2)^m C^{m} \bigg( \int_{\mathbb{T}^{N} } \Big\Vert  \sum_{\vert \alpha \vert = m}  x_{\alpha}  z^{\alpha} \Big\Vert^{2} dz  \Bigg)^{1/2}.
%\end{eqnarray*}
\[
 \bigg(  \sum_{\alpha}  \Vert x_{\alpha} \Vert^{q} \bigg)^{2/q}  = \bigg(  \sum_\alpha  \Vert x_{\alpha} z^\alpha \Vert^{q} \bigg)^{2/q} \le  C_q^2   \int_{\Omega} \Big\Vert \sum_\alpha \varepsilon_{\alpha} (\omega) x_{\alpha} z^\alpha \Big\Vert^{2} d \omega
  %\\ &\le & C_q (\sqrt 2)^m \, \int_{\Omega} \Big\Vert \sum_\alpha \varepsilon_{k} (\omega) x_{\alpha} z^\alpha \Big\Vert d \omega ,
\]
%where the polynomial Kahane  inequality (Proposition~\ref{pol-kahane})  was used in the last step.
Integrating this inequality on $z\in \mathbb{T}^{N}$ and using \eqref{estreshita}, we obtain
\[
 \bigg(  \sum_\alpha  \Vert x_{\alpha} \Vert^{q} \bigg)^{2/q} \le    C_q^2   \int_{\mathbb{T}^{N} }   \int_{\Omega} \Big\Vert \sum_\alpha \varepsilon_{\alpha} (\omega) x_{\alpha} z^\alpha \Big\Vert^2 d \omega  \, dz
%&\le &  C_q^2    \bigg(  \int_{\Omega} \int_{\mathbb{T}^{N} } \Big\Vert \sum_\alpha
 % \varepsilon_{k} (\omega) x_{\alpha} z^\alpha \Big\Vert^2  dz  \, d \omega  \Bigg)^{1/2} \\
 \le   C_q^2  C^{2m}  \int_{\mathbb{T}^{N} } \Big\Vert  \sum_{\alpha}  x_{\alpha}  z^{\alpha} \Big\Vert^{2} dz  \,.
\]
Therefore, $X$ has hypercontractive homogeneous cotype $q$.
 \end{proof}

In the next result we follow and adapt some of the ideas of \cite{Pi78}. We recall that an operator between Banach spaces $u : X \to Y$ is absolutely $q$-summing if there is $C>0$ such that for every finite family
$x_{1}, \ldots , x_{n} \in X$ we have
\[
\bigg( \sum_{j=1}^{n} \|ux_{j}\|^q \bigg)^{\frac{1}{q}}
\leq C \sup_{x^{*} \in B_{X^{*}} } \bigg( \sum_{j=1}^{n} | x^{*} (x_{j})|^q \bigg)^{\frac{1}{q}}
\,.
\]
The best constant $C$ in this inequality is called the absolutely $q$-summing norm of $u$ and is denoted by $\pi_{q}(u)$.

\begin{proposition} \label{rubio}
If $X$ has cotype $q$ and l.u.st., then there exists $C>0$, such that for every choice of finitely many  $x_{\alpha} \in X$ and signs $\varepsilon_{\alpha} = \pm 1$
\[
 \bigg( \int_{\mathbb{T}^{N} } \Big\Vert  \sum_{\vert \alpha \vert = m}  x_{\alpha} \varepsilon_{\alpha} z^{\alpha} \Big\Vert^{2} dz  \Bigg)^{1/2}
\leq C \, q^{m/2} \bigg( \int_{\mathbb{T}^{N} } \Big\Vert  \sum_{\vert \alpha \vert = m}  x_{\alpha}  z^{\alpha} \Big\Vert^{2} dz  \Bigg)^{1/2}.
\]
In particular, $X$ satisfies \eqref{estreshita}.
\end{proposition}
\begin{proof}
%We consider the mappings $\Phi: \mathbb{T}^{N} \to L_q(\mathbb T^N)$ and $\Psi: \mathbb{T}^{N} \to L_1(\mathbb T^N)$ given by $$\Phi(w)(z) = \sum_{\vert \alpha \vert = m}  \varepsilon_{\alpha} x_{\alpha} w^\alpha z^{\alpha} \quad \text{and} \quad
%\Psi(w)(z) = \sum_{\vert \alpha \vert = m}  x_{\alpha}  w^\alpha z^{\alpha}.$$
We fix $\varepsilon_{\alpha} = \pm 1$ for each $\alpha \in \mathbb{N}_{0}^{N}$ with $\vert \alpha \vert=m$, and define operators $u:X^*\to L_q(\mathbb T^N)$ and $v:X^*\to L_1(\mathbb T^N)$  by
\[
u(x^*) (z) = \sum_{\vert \alpha \vert = m} \varepsilon_{\alpha} x^*(x_{\alpha})  z^{\alpha} \quad \text{and} \quad v(x^*) (z) = \sum_{\vert \alpha \vert = m} x^*(x_{\alpha})   z^{\alpha} \,.
\]
For each $x^*\in X^*$, the scalar case in Proposition
\ref{pol-kahane} gives
\begin{align*}
\|u(x^*)\|_{L_q} & =  \left( \int_{\mathbb T^N} \Big| \sum_{\vert \alpha \vert = m} \varepsilon_{\alpha} x^*(x_{\alpha})  z^{\alpha}\Big|^q  dz \right)^{\nicefrac 1 q} \le  q^{\nicefrac m 2} \int_{\mathbb T^N} \Big| \sum_{\vert \alpha \vert = m} \varepsilon_{\alpha} x^*(x_{\alpha})  z^{\alpha}\Big| dz
 \\
&= q^{\nicefrac m 2} \|v(x^*)\|_{L_1} \,.
\end{align*}
From this and the very  definition  of the absolutely $1$-summing norm we easily deduce that  $\pi_1(u)\le  q^{\nicefrac m 2} \pi_1(v)$. By \cite[Theorem~1.1]{Pi78} we have
\[
\pi_q(^tu)\le C\pi_1(u)\le C q^{\nicefrac m 2}\pi_1(v) \,.
\]

Now, \cite[Proposition~1.1]{Pi78} states that, for $1\le p\le \infty$, every
 $\varphi_1, \dots ,  \varphi_n\in L_p(\mathbb T^N)$ (or any other $L_p(\mu)$,\, $\mu$ a probability measure) and every $y_1,\dots,y_n \,\in X$, the operator $S: X^*\to L_p(\mathbb T^N)$ given by \begin{equation}\label{operador}S(x^*)=\sum _{i=1}^n x^*(y_i) \varphi_i\end{equation}
  satisfies
  \begin{equation}\label{desigualdad}
  \pi_p(S)\le \left( \int_{\mathbb T^N} \Big\| \sum _{i=1}^ny_i \varphi_i(z) \Big\|^p dz\right)^{\nicefrac 1 p}  \le \pi_p(S^t).
  \end{equation}

Note that we can write $u$ and $v$ as in \eqref{operador}, taking $\varphi_\alpha(z)=\varepsilon_\alpha z^\alpha$, and $\varphi_\alpha(z)=z^\alpha$ respectively. As a consequence, we can use the second inequality in \eqref{desigualdad} for $u$ and the first inequality in \eqref{desigualdad} for $v$ to obtain
\begin{multline*}
\bigg( \int_{\mathbb{T}^{N} } \Big\Vert  \sum_{\vert \alpha \vert = m}  x_{\alpha} \varepsilon_{\alpha} z^{\alpha} \Big\Vert^{2} dz  \Bigg)^{1/2}  \le  \bigg( \int_{\mathbb{T}^{N} } \Big\Vert  \sum_{\vert \alpha \vert = m}  x_{\alpha} \varepsilon_{\alpha} z^{\alpha} \Big\Vert^{q} dz  \Bigg)^{1/q} \\ \le \pi_q(^tu)
\le C q^{\nicefrac m 2}\pi_1(v) \le C \, q^{m/2} \int_{\mathbb{T}^{N} } \Big\Vert  \sum_{\vert \alpha \vert = m}  x_{\alpha}  z^{\alpha} \Big\Vert dz   \\
\le C \, q^{m/2} \bigg( \int_{\mathbb{T}^{N} } \Big\Vert  \sum_{\vert \alpha \vert = m}  x_{\alpha}  z^{\alpha} \Big\Vert^{2} dz  \Bigg)^{1/2}\qedhere
\end{multline*}
\end{proof}

\medskip

\subsection{Fourier cotype implies hypercontractive homogeneous cotype}\label{subsec-fourier}
Now we show that Banach spaces with Fourier cotype also have hypercontractive homogeneous cotype. This is independent from our result in the previous section (Theorem~\ref{lust}), since for example the Schatten
classes $\mathcal{S}_{r}$ have Fourier cotype but do not have l.u.st. \\
There are many equivalent definitions of Fourier cotype (see \cite{GCuKaKoTo98}). Let us give the one that is more akin to our framework. Given $2 \leq q < \infty$, we say that $X$ has Fourier cotype $q$
if there is a constant $C > 0$
such that for each choice of finitely many vectors $x_1, \ldots, x_N \in X$ we have
\begin{equation}\label{fcotype}
\Big(\sum_{k=1}^N \big\| x_k \big\|^q \Big)^{1/q}  \leq C \Big( \int_{\mathbb{T}} \Big\| \sum_{k=1}^N  x_k z^{k}  \Big\|^{q'} dz \Big)^{1/q'} \,.
\end{equation}
We write
\[
\fct(X) := \inf \Big\{2 \le q< \infty \,\big|\, X \,\,\, \text{has Fourier cotype} \,\,\, q \Big\}\,,
\]
and (although this infimum in general is not attained) we call it the optimal Fourier cotype of $X$.
In the literature (see, for example, \cite{Ma13}) the sums in \eqref{fcotype} usually run from $-M$ to $M$ or in $\mathbb Z$, but it is not hard to check that both definitions are equivalent: the rotation invariance of the measure $dz$ gives
$$ \Big( \int_{\mathbb{T}} \Big\| \sum_{j=-M}^M  x_j \, z^{j}  \Big\|^{q'} dz \Big)^{1/q'} = \Big( \int_{\mathbb{T}} \Big\|  z^{M}  \sum_{j=-M}^M  x_j \,z^{j} \Big\|^{q'} dz \Big)^{1/q'} = \Big( \int_{\mathbb{T}} \Big\| \sum_{k=0}^{2M}  x_{k-M}\, z^{k} \Big\|^{q'} dz \Big)^{1/q'},$$from which the equivalence follows easily.

Spaces with Fourier cotype satisfy a stronger version of hypercontractive homogeneous cotype, with a uniform constant for every (homogeneous or not) polynomial of any degree. This result can be seen as a consequence of, for example,
\cite[Theorem 6.14]{GCuKaKoTo98} and the equivalence between Fourier type $p$ and Fourier cotype $q$ when $\frac 1 p + \frac 1 q =1$.

\begin{proposition}\label{fourier-implies-pol}
Let $X$ be a Banach space with Fourier  cotype $q\ge 2$, then there exists $C >0$ such that for every finite family $(x_{\alpha})_{\alpha  \in \mathbb{N}_{0}^{(\mathbb{N})}}$ we have
\[
 \bigg( \sum_{ \alpha  } \Vert x_{\alpha} \Vert^{q} \bigg)^{1/q} \leq C \bigg( \int_{\mathbb{T}^{N}} \Big\Vert  \sum_{ \alpha }  x_{\alpha} z^{\alpha} \Big\Vert^{q'} dz \bigg)^{1/q'} \,.
\]
In particular, $X$ has hypercontractive homogeneous cotype $q$.
\end{proposition}
\begin{proof}Let $m$ be the maximum of all $\alpha_j$'s such that $x_\alpha$ is not zero.
By the rotation invariance of the measures $dz_2,\dots, dz_N$, fixed $z_1\in \mathbb T$ we have:
\begin{align*} &\int_{\mathbb{T}^{N-1}} \Big\Vert  \sum_{ \alpha }  x_{\alpha} z_1^{\alpha_1} z_2^{\alpha_2} \cdots z_N^{\alpha_N}\Big\Vert^{q'} dz_2 \cdots dz_N \\ & =  \int_{\mathbb{T}^{N-1}} \Big\Vert  \sum_{ \alpha }  x_{\alpha} z_1^{\alpha_1} (z_2\,z_1^{m+1})^{\alpha_2} \cdots ( z_N\, z_1^{(m+1)^{N-1}})^{\alpha_N}\Big\Vert^{q'} dz_2 \cdots dz_N \\
& = \int_{\mathbb{T}^{N-1}} \Big\Vert  \sum_{ \alpha }  x_{\alpha} z_1^{\alpha_1+(m+1)\alpha_2 + \cdots + (m+1)^{N-1} \alpha_N} z_2^{\alpha_2} \cdots z_N^{\alpha_N}\Big\Vert^{q'} dz_2 \cdots dz_N.
\end{align*}
As a consequence, a change in the order of integration gives
\begin{align}
\label{integral-four-prin} &\int_{\mathbb{T}^{N}}  \Big\Vert  \sum_{ \alpha }  x_{\alpha} z^{\alpha} \Big\Vert^{q'} dz \\
& = \int_{\mathbb{T}^{N-1}} \left( \int_{\mathbb T} \Big\Vert  \sum_{ \alpha } \left( x_{\alpha} z_2^{\alpha_2} \cdots z_N^{\alpha_N} \right) z_1^{\alpha_1+(m+1)\alpha_2 + \cdots + (m+1)^{N-1} \alpha_N} \Big\Vert^{q'} dz_1 \right) dz_2 \cdots dz_N.\label{integral-four}
\end{align}
For every  $\alpha$ for which $x_\alpha$ is not zero we have $0\le \alpha_j \le m$, $j=1,\dots, N$. Also, if a multi index $\beta$ satisfies $0\le \beta_j \le m$, $j=1,\dots,N$ and
$$\alpha_1+(m+1)\alpha_2 + \cdots + (m+1)^{N-1} \alpha_N = \beta_1+(m+1)\beta_2 + \cdots + (m+1)^{N-1} \beta_N,$$ then we must have $\alpha=\beta$ (this is just the uniqueness of the expansion of a natural number in base $m+1$). Therefore, the powers of $z_1$ in the sum in \eqref{integral-four} are all different. We can then apply \eqref{fcotype} to the inner integral of \eqref{integral-four} for each fixed $z_2,\dots,z_N$. This gives  that the whole expression in \eqref{integral-four} is bounded from below by
$$
 \frac 1 {C^{q'}} \int_{\mathbb{T}^{N-1}}  \left( \sum_{ \alpha } \left\| x_{\alpha} z_2^{\alpha_2} \cdots z_N^{\alpha_N} \right\|^q \right)^ {q'/q} dz_2 \cdots dz_N \\
 = \frac 1 {C^{q'}} \left( \sum_{ \alpha } \left\| x_{\alpha} \right\|^q \right)^ {q'/q}.$$
So  \eqref{integral-four-prin} is bounded from below by this last expression, which is the result we were looking for.
\end{proof}

Let us point out two things before we go on. First, note that in Proposition~\ref{fourier-implies-pol} we have a cotype-like inequality that holds for any polynomial of any degree and on any number of variables. Following our
philosophy  we could call this \emph{analytic cotype}. Second, observe that if a Banach space satisfies the inequality in Proposition~\ref{fourier-implies-pol}, then it obviously satsifies \eqref{fcotype}. Hence, Proposition~\ref{fourier-implies-pol} is actually an `if and  only if' result.\\

Let us recall that a Banach space satisfying the reverse inclusion in \eqref{cotype} is said to have type $q$.  It is a well  known fact (which follows, for example, from \cite[Section 6.1.8.6]{Pie07}) that if $X$ has type 2 and cotype $q_0$, then it has Fourier cotype $q$ for every $q>q_0$. Therefore, we have
\[
\ct(X)=\fct(X) = \ctp(X)
\]
for every Banach space $X$ with type 2.

\subsection{Examples}\label{subsec_examples}

By Theorem~\ref{lust}, cotype and hypercontractive homogeneous cotype coincide in $L_r(\mu)$ and, more generally, in $\mathcal L_r$-spaces for $1\le r\le \infty$ (see Chapters 3 and 17 in \cite{DiestelJarchowTonge} for the definition of $\mathcal L_r$-spaces and their local unconditional strucuture, respectively).
As a consequence, a $\mathcal L_r$-space $X$  have hypercontractive homogeneous cotype $\ct(X) =\max\{2,r\}$ for $1\le r\le \infty$.

The Schatten classes $\mathcal{S}_r$ (as well as $\mathcal L_r$-spaces) have Fourier cotype $\max\{r,r'\}$ and these are the optimal values (see \cite[Theorem~1.6]{GCMaPa03} or \cite[Theroem~6.8]{GCuPar04}). Thus
by Proposition~\ref{fourier-implies-pol}, they have hypercontractive homogeneous
cotype $\max\{r,r'\}$  (in fact, they have the much stronger uniform and non-homogeneous one given in Proposition~\ref{fourier-implies-pol}). On the other hand, these spaces have cotype $\max \{2,r\}$ and type $\min\{2,r\}$ \cite[page~228]{DiestelJarchowTonge}. In other words, hypercontractive homogeneous and usual cotype coincide for Schatten clases for $r\ge 2$. Note that, since Schatten classes with $r \neq 2$ do not have l.u.st. \cite[page~364]{DiestelJarchowTonge}, Theorem~\ref{lust} does not apply in this case.

We summarize these positive answers to our conjecture in the following

\begin{corollary}\label{coincide}
Cotype and hypercontractive homogeneous cotype coincide in $\mathcal L_r$-spaces for $1\le r \le \infty$ and in $\mathcal{S}_r$ for $2\le r \le \infty$.
\end{corollary}

\section{Sets of monomial convergence for $ H_{p}(\mathbb{T}^{\infty},X)$}
\label{sec-applications}
Each function $f \in H_{p}(\mathbb{T}^{\infty},X)$ defines a formal power series $ \sum_{\alpha}  \hat{f}(\alpha) z^{\alpha}$. We address now the question of for which $z$'s does this power series converge.
Given a Banach space $X$ and $1 \leq p \leq \infty$, we define the set of monomial convergence of
$ H_{p}(\mathbb{T}^{\infty},X)$:
\begin{equation*} \label{mon Hardy}
 \mon H_{p}(\mathbb{T}^{\infty},X) = \Big\{ z \in \mathbb{C}^\mathbb{N}\,\,\,\Big|\,\, \, \sum_{\alpha} \| \hat{f}(\alpha) z^{\alpha} \|_X < \infty
\text{ for all }  f \in  H_{p}(\mathbb{T}^{\infty},X) \Big\} \,.
\end{equation*}
We also define, for each $m \in \mathbb{N}$,
\begin{equation*} \label{m-mon Hardy}
 \mon H_{p}^m(\mathbb{T}^{\infty},X) = \Big\{ z \in \mathbb{C}^\mathbb{N}\,\,\,\Big|\,\, \, \sum_{\alpha} \| \hat{f}(\alpha) z^{\alpha} \|_X < \infty
\text{ for all }  f \in  H_{p}^m(\mathbb{T}^{\infty},X) \Big\} \,,
\end{equation*}
where
\begin{equation*} \label{Hpm}
 H_{p}^{m}(\mathbb{T}^{\infty}) = \Big\{ f \in H_{p} (\mathbb{T}^{\infty})\,\,\,\Big|\,\, \, \hat{f} (\alpha) \neq 0 \, \Rightarrow \, \vert \alpha \vert = m
 \Big\}\,.
\end{equation*}
The problem of determining $\mon H_{p}(\mathbb{T}^{\infty})$ and  $\mon H_{p}^m(\mathbb{T}^{\infty})$
in the scalar-valued case goes back to Bohr at the 1910's, and the so far  most general result was recently given in \cite{BayartDefantFrerickMaestreSevilla2014} (for more information and detailed  historical remarks see the references therein): For $p =\infty$ we
have
\[
 \mathbf{B} \,\subset \,\mon H_{\infty}(\mathbb{T}^{\infty}) \,\subset\, \mathbf{\overline{B}}\,,
 \]
where
\begin{eqnarray*}
B& =& \Bigg\{
			u \in B_{c_0}
		\,\,\,\Big|\,\, \,
			\limsup_n \frac{1}{\log n} \sum_{k=1}^n \lvert u_k^\ast \rvert^2 < 1
		 \Bigg\}\,\\
\overline B &=&  \Bigg\{
			u \in B_{c_0}
		\,\,\,\Big|\,\, \,
			\limsup_n \frac{1}{\log n} \sum_{k=1}^n \lvert u_k^\ast \rvert^2 \leq 1
		 \Bigg\}
\end{eqnarray*}
($u^\ast$ the decreasing rearrangement of $u$), and for $1 \leq p < \infty$
\[
\mon H_{p}(\mathbb{T}^{\infty}) = \ell_{2} \cap  B_{c_{0}} \,\,\,\text{ for }\,\,\,1 \leq p < \infty\,.
\]
In  the homogeneous case we have for each $m$
\[
  \mon H_{p}^{m} (\mathbb{T}^{\infty}) =
\begin{cases}
\ell_{\frac{2m}{m-1}, \infty} & \text{ for }  p = \infty\\
  \ell_{2}  & \text{ for } 1 \leq p < \infty\,.
  \end{cases}
\]
It can be seen easily that in the preceding results scalar-valued functions can be replaced by functions with
values in finite dimensional Banach spaces -- but the following theorem indicates   that the situation for functions with
have their range in infinite dimensional spaces is substantially different (see also \cite{DeSe12}).

\begin{theorem} \label{verynice}
Let $1 \leq p \leq \infty$, $m \in \mathbb{N}$, and $X$ an infinite dimensional Banach space $X$.
\begin{enumerate}
\item \label{verynice1} If $X$ has trivial cotype, then
\[
\mon H_p(\mathbb{T}^\infty,X)  =\ell_1  \cap B_{c_0} \,\, \text{ and } \,\, \mon H_p^m(\mathbb{T}^\infty,X) = \ell_{1} \,  .
\]
\item \label{verynice2} If $X$ has hypercontractive homogeneous cotype $\ct(X)< \infty$, then
\[
\mon H_p(\mathbb{T}^\infty,X)  = \ell_{\ct(X)'} \cap B_{c_0} \cap   \,\, \text{ and } \,\,  \mon H_p^m(\mathbb{T}^\infty,X) =   \ell_{\ct(X)'} \,.
\]
\end{enumerate}
\end{theorem}

To see some examples, we have  mentioned in Section~\ref{subsec_examples} that a $\mathcal L_r$-space $X$
has hypercontractive homogeneous cotype $\ct(X) =\max\{2,r\}$ for $1\le r\le \infty$, and that for $r\ge 2$, $\mathcal{S}_r$ has hypercontractive homogeneous cotype $\ct(\mathcal{S}_r)=r$. As a consequence, we have the following.

\begin{corollary} \label{nice}  Let $1 \leq p \leq \infty$.
\begin{enumerate}
\item  If\, $1 \leq r \leq \infty$ and $X$ is a $\mathcal L_r$-space
then
\[
\mon H_p(\mathbb{T}^\infty,X)  =  \ell_{\min \{2,r' \}}   \cap B_{c_0} \,\, \text{ and } \,\,  \mon H_p^m(\mathbb{T}^\infty,X) =  \ell_{\min \{2,r' \}}  \, .
\]
\item If \,$2 \leq r \leq \infty$,
then
\[
\mon H_p(\mathbb{T}^\infty,\mathcal{S}_r)  =  \ell_{r'}  \cap B_{c_0}   \,\, \text{ and } \,\,  \mon H_p^m(\mathbb{T}^\infty,\mathcal{S}_r) =  \ell_{r'}  \,.
\]
\end{enumerate}
\end{corollary}

%---------------------------
%
%To see some examples, we have already mentioned that for $2\le r\le \infty$ the spaces  $X=L_r (\mu)$ and $X=\mathcal{S}_r$ have hypercontractive homogeneous cotype $\ctp(X)  = \fct(x) =\ct(X)=  r$.  This gives the following
%
%\begin{corollary} \label{nice}
%Let $1 \leq p \leq \infty$ and $2\le r \le \infty$. Then,\\
%\[
%\mon H_p(\mathbb{T}^\infty,\ell_r)  = \mon H_p^m(\mathbb{T}^\infty,\ell_r) =   B_{c_0} \cap  \ell_{r'}
%\]
%\\
%and
%\[
%\mon H_p(\mathbb{T}^\infty,\mathcal{S}_r)  = \mon H_p^m(\mathbb{T}^\infty,\mathcal{S}_r) =   B_{c_0} \cap  \ell_{r'}.
%\]
%\end{corollary}

We split the proof of Theorem~\ref{verynice} in two steps, that we present as separate lemmas. Before we state the first one, let us recall (see e.g. \cite[Chapter~14]{DiestelJarchowTonge}) that a Banach space $X$ finitely factors $\ell_q \hookrightarrow \ell_\infty$  with $2 \leq q \le \infty$ whenever for each  $N$ there are vectors
$x_1, \ldots,x_N \in X$ such that for every choice of $\lambda_1, \ldots, \lambda_N \in \mathbb{C}$ we have
\begin{equation} \label{cassis}
\frac{1}{2}  \|\lambda\|_\infty  \leq \big\| \sum_{n=1}^N \lambda_k x_k\big\| \leq  \|\lambda\|_q\,.
\end{equation}

\begin{lemma} \label{lema1}
Let $X$ be an infinite dimensional Banach space which  finitely factors $\ell_q \hookrightarrow \ell_\infty$. Then
$\mon H_\infty^1 (\mathbb{T}^\infty, X) \subset \ell_{q'}$\,.
\end{lemma}
\begin{proof}
Let us take $z \in  \mon H_\infty^1(\mathbb{T}^\infty, X)$. By a standard closed graph argument there  is a constant $c(z)>0$ such that for each $f \in  H_\infty^1 (\mathbb{T}^\infty, X) $ we have
 \[
 \sum_{k=1}^\infty \|f(e_k)\| \,|z_k| \leq c(z) \|f\|_\infty.
 \]
We fix some $N \in \mathbb{N}$ and choose $x_1, \ldots, x_N \in X$ as is \eqref{cassis}. %such that for each $\lambda \in \mathbb{C}^N$
%\[
%\frac{1}{2} \|\lambda\|_\infty  \leq \big\| \sum_{k=1}^N \lambda_k x_k \big\|  \leq\|\lambda\|_q\,.
%\]
Given $w_1, \ldots, w_N \in \mathbb{C}$ we define $f \in H_\infty^1 (\mathbb{T}^\infty, X)$ by $f(u)= \sum_{k=1}^N (x_kw_k) \, u_k$\,. Then  we have
\begin{multline*}
\sum_{n=1}^N |w_n z_n|
\leq 2 \sum_{n=1}^N \|(w_nx_n) z_n\|
\leq 2 c(z) \sup_{u \in \mathbb{T}^\infty}  \big\|  \sum_{n=1}^N  (w_nx_n) u_n \big\|
\\
\leq 2 c(z) \sup_{u \in \mathbb{T}^\infty}  \big\|(w_nu_n)_{n=1}^N\big\|_{q}
\leq 2 c(z)  \big\|(w_n)_{n=1}^N\big\|_{q}\,.
\end{multline*}
Since the $w_1, \ldots, w_N$ are arbitrary, this clearly proves the claim.
\end{proof}

\begin{lemma} \label{lema2}
If $X$ has hypercontractive homogeneous cotype $q$, then $\ell_{q'} \cap B_{c_0} \subset \mon H_1 (\mathbb{T}^\infty, X)$\,.
\end{lemma}
\begin{proof}
Assume here that $q < \infty$. We first prove that there is a constant $C>0$ such that for each $m$, each  $f \in H_1^m(\mathbb{T}^\infty, X)$, and each $y \in \ell_{q'} \cap  B_{c_0} $  we have
\[
\sum_{|\alpha|=m} \|\widehat{f}(\alpha)y^\alpha\|
\leq
C^m \Big( \sum_{|\alpha|=m}  |y^{\alpha}|^{q'}  \Big)^{1/q'} \, \|f\|_{1}\,.
\]
We fix  such $f,y$ and $N \in \mathbb{N}$; proceeding as in \cite[page~524]{CarandoDefantSevilla2013} we can find a function $f_N \in H_1(\mathbb{T}^N, X)$ such that $\|f_N\|_1 \leq \|f\|_1$ and $\widehat{f_N}(\alpha) =
\widehat{f}(\alpha)$ for all $\alpha \in \mathbb{N}_0^N$. Using this fact, that $X$ has hypercontractive homogeneous cotype $q$ (with constant $D$, say) and
Proposition~\ref{pol-kahane} (the polynomial Kahane inequality) we have  for $C=D \sqrt{2}$
\begin{equation} \label{eroica}
\begin{split}
\sum_{\overset{\alpha \in \mathbb{N}_0^N}{|\alpha|=m}} \|\widehat{f}(\alpha)y^\alpha\|
&
\leq
\Big( \sum_{|\alpha|=m}  |y^{ \alpha}|^{q'}  \Big)^{1/q'}
\big( \sum_{|\alpha|=m}  \|\widehat{f}_N(\alpha)\|^q  \big)^{1/q}
\\&
\leq D^m  \Big( \sum_{|\alpha|=m} |y^{ \alpha}|^{q'}  \Big)^{1/q'} \Big(  \int_{\mathbb{T}^N}  \|f_N(y)\|_X^2 dz\Big)^{1/2}
\\&
\leq D^m \sqrt{2}^m\Big( \sum_{|\alpha|=m}  |y^{\alpha}|^{q'}  \Big)^{1/q'} \, \|f_N\|_{1}
\leq C^m \Big( \sum_{|\alpha|=m}  |y^{\alpha}|^{q'}  \Big)^{1/q'} \, \|f\|_{1}
\,,
\end{split}
\end{equation}
Take  now $0 < r < 1/C\,,$ and let us check that
\[
r B_{\ell_{q'}} \cap B_{c_0} \subset \mon H_1(\mathbb{T}^\infty, X) \,.
\]
To do this we fix some  $f \in H_1(\mathbb{T}^\infty, X)$ and $z=ry \in r B_{\ell_{q'}} \cap B_{c_0}$. For each $N$ we consider $f_N$ as above. By \cite[Proposition~2.5]{CarandoDefantSevilla2013} there is
$f_N^m \in  H_p(\mathbb{T}^N, X)$
such that $\widehat{f_{N}}(\alpha) = \widehat{f_{N}^m}(\alpha)$ for all  $\alpha \in \mathbb{N}_0^N$ with $|\alpha|=m$, $\widehat{f_{N}^m}(\alpha)=0$ if $\vert \alpha \vert \neq m$, and
$\|f_{N}^m \|_1 \leq \| f_{N}\|_1$\,.
 Then, applying \eqref{eroica} to $f_{N}^{m}$ we get
\begin{align*}
 \sum_{\substack{\alpha \in \mathbb{N}_0^N}} \|\widehat{f}(\alpha) z^\alpha\|
 & =
  \sum_{m=0}^\infty \sum_{\substack{\alpha \in \mathbb{N}_0^N\\ |\alpha|=m}} \|\widehat{f_N^m}(\alpha)\,\, (ry)^\alpha\|
  \\&
  \leq
  \sum_{m=0}^\infty r^m\sum_{\substack{\alpha \in \mathbb{N}_0^N\\ |\alpha|=m}} \|\widehat{f_N^m}(\alpha)\,\, y^\alpha\|
  \\&
  \leq
  \sum_{m=0}^\infty r^m C^m
  \Big( \sum_{|\alpha|=m}  |y^{\alpha}|^{q'}  \Big)^{1/q'} \|f_m^N\|_{1}
  \\&
  \leq
  \sum_{m=0}^\infty r^m C^m
  \Big( \sum_{|\alpha|=m}  |y^{\alpha}|^{q'}  \Big)^{1/q'} \|f\|_1
   \\&
  \leq
   \Big( \sum_{\alpha}  |y^{\alpha}|^{q'}  \Big)^{1/q'} \|f\|_1
  \sum_{m=0}^\infty r^m C^m\,.
\end{align*}
Let us recall (see e.g. \cite[page~24]{DeMaPr09}) that
\begin{equation} \label{obelisco}
z \in \ell_1 \cap B_{c_0} \,\, \text{ if and only if } \sum_{\alpha \in \mathbb{N}_{0}^{(\mathbb{N})}} |z^\alpha|< \infty \,.
\end{equation}
This implies that the last term is finite.

We can now complete our argument. For  $z \in \ell_{q'} \cap B_{c_0}$ we choose $n_0$   such that
\[
\tilde{z}= (0, \ldots, 0, z_{n_0}, z_{n_0+1}, \ldots) \in r B_{\ell_{q'}} \cap B_{c_0}\,.
\]
Then  $\tilde{z} \in \mon H_1(\mathbb{T}^\infty, X) $, and a straightforward vector-valued extension of \cite[Lemma~3.7]{BayartDefantFrerickMaestreSevilla2014} (see also  \cite[Lemma~2]{DeGaMaPG08} where the analogous result for $\mon H_{\infty}(\mathbb{T}^\infty, X)$ is shown) completes the proof.
\end{proof}

With this at our hand we are now ready to prove our main result in this section.
\begin{proof}[Proof of Theorem~\ref{verynice}]
We show parts (1) and (2) together.
Take  $1 \leq p \leq \infty$ and  assume that $X$ is an infinite dimensional Banach space with hypercontractive homogeneous cotype $\ct(X)$.
By a vector-valued modification of \cite[Lemma~3.3]{BayartDefantFrerickMaestreSevilla2014}
we have  $$\mon H^{m}_p (\mathbb{T}^\infty, X) \subset \mon H_p^{m-1} (\mathbb{T}^\infty, X)$$
and trivially
\[
\mon H^{1}_p (\mathbb{T}^\infty, X) \subset \mon H^1_{\infty} (\mathbb{T}^\infty, X)\,.
\]
First of all, as a consequence of a deep result of Maurey and Pisier \cite{MaPi76} (see also \cite[Theorem~14.5 and page~304]{DiestelJarchowTonge}) $X$ always finitely factors $\ell_{\ct(X)} \hookrightarrow \ell_\infty$.
Then Lemmas~\ref{lema1} and \ref{lema2} give
\[
\ell_{\ct(X)'} \cap B_{c_0} \subset \mon H_1(\mathbb{T}^\infty, X) \subset \mon H_p(\mathbb{T}^\infty, X)
\subset \mon H_\infty^1 (\mathbb{T}^\infty, X)  \cap B_{c_0} \subset \ell_{\ct(X)'} \cap B_{c_0}\,.
\]
This completes the argument.
\end{proof}

Let us remark that in Theorem~\ref{verynice}--\eqref{verynice2} we are assuming that $X$ has non-trivial hypercontractive homogeneous cotype (hence also usual cotype) and both optimal values are equal and attained. If this is not the case, then our proof
shows that
\begin{equation}\label{no-assump}
 \ell_{\ct(X)'} \cap B_{c_0} \subset\mon H_p(\mathbb{T}^\infty,X)  \subset \mon H_p^m(\mathbb{T}^\infty,X)  \cap B_{c_0} = \ell_{\ctp(X)'+ \varepsilon}  \cap B_{c_0}
\end{equation}
for all $\varepsilon >0$.

\section{Multiplicative $\ell_1$-multipliers for Hardy spaces of Dirichlet series}

Power series in infinitely many variables and Dirichlet series can be identified by an ingenious idea of Bohr.  For a fixed Banach space $X$ we denote by $\mathfrak{P}(X)$  the vector space  of all formal power series
$\sum_{\alpha} c_{\alpha} z^{\alpha}$ in $X$ and by $\mathfrak{D}(X)$ the vector space  of all Dirichlet series $\sum_n a_{n} n^{-s}$ in $X$ . Let $(p_{n})_{n}$ be the sequence of prime numbers. Since each integer $n$ has a unique prime number decomposition
$n=p_{1}^{\alpha_{1}} \cdots p_{k}^{\alpha_{k}} = p^{\alpha}$ with $\alpha_j \in \mathbb{N}_0, \, 1 \le j \leq k$, the linear mapping, that we call the Bohr transform in $X$,
\begin{align*}
\mathfrak{B}_X: \mathfrak{P}(X) & \longrightarrow \mathfrak{D}(X) \,\,, \,\,\,\,\,\,
\textstyle\sum_{\alpha \in \mathbb{N}_{0}^{(\mathbb{N})}} c_{\alpha} z^{\alpha}
\rightsquigarrow \textstyle\sum_{n=1}^{\infty} a_{n} n^{-s}\,, \,\,\, \text{where}\,\,\, a_{p^{\alpha}} = c_{\alpha}\,
\end{align*}
is bijective.
Given $1 \leq p \leq \infty$ and $m \in \mathbb{N}$, define the two linear spaces
\[
\mathcal{H}_p(X)  = \mathfrak{B}_X \left(H_p(\mathbb{T}^m,X)\right)
\]
and
\[
\mathcal{H}^m_p(X)  = \mathfrak{B}_X \left(H^m_p(\mathbb{T}^m,X)\right)\,,
\]
which through the norms induced by $\mathfrak{B}_X$ form (what we call) the Banach spaces of vector-valued Hardy-Dirichlet series.

A scalar sequence $(b_n)$ is  called multiplicative (or completely multiplicative) if
$b_{mn}=b_nb_m$ for all $m,n$\,. Basic examples of multiplicative sequences $(b_n)$ are the sequences $1/n^\sigma$. We call a scalar sequence $(b_n)$ an  $\ell_1$-multiplier for  $\mathcal{H}_p(X)$  whenever for all $\sum_n a_n n^{-s} \in \mathcal{H}_p(X)$  we have
\[
\sum_{n=1}^\infty \|a_n \|_X\,\, |b_n| \le \infty \,\,\,\,\, \text{ for all } \,\,\,\,\, \sum_n a_n n^{-s} \in \mathcal{H}_p(X).
\]
All multiplicative $\ell_1$-multipliers for  $\mathcal{H}_p(X)$ are denoted
\[
\mult\, \mathcal{H}_p(X)\,,
\]
and, given $m \in \mathbb{N}$,  in the homogenous case  of course an analogous definition
\[
\mult\, \mathcal{H}^m_p(X)
\]
can be done.
In \cite[Remark~4.1]{BayartDefantFrerickMaestreSevilla2014} (here again  the scalar case immediately transfers  to the vector valued case) we have the following link between sets of monomial convergence and multiplicative $\ell_1$-multipliers.
\begin{remark} \label{multi}
Let $(b_n)$ be a multiplicative sequence of complex numbers, and $1 \le p \le \infty$. Then $(b_n)$ is an $\ell_1$-multiplier for  $\mathcal{H}_p(X)$  if and
only if
$(b_{p_k}) \in \mon H_p(\mathbb{T}^\infty,X)$. Clearly, an analogous equivalence holds whenever  we replace
 $\mathcal{H}_p(X)$ by $\mathcal{H}^m_p(X)$.
\end{remark}

\begin{remark} \label{malti}
Let us observe that if $b \in  \ell_{p}$ is multiplicative, then $\vert b_{n} \vert < 1$ for all $n$. Indeed, if some $\vert b_{n} \vert \geq 1$, then since the sequence is multiplicative $\vert b_{n^{k}}\vert \geq 1$ for every $k$ and this contradicts the fact that $b$ is in $\ell_{p}$. Then necessarily
$b \in B_{c_{0}}$ and by \eqref{obelisco} we have that $(b_{p_{k}})_{k} \in \ell_{p}$ if and only if $b \in \ell_{p}$.
\end{remark}

With Remarks~\ref{multi},~\ref{malti}  and Theorem~\ref{verynice} we immediately have the following characterization of multiplicative $\ell_1$-multipliers of $\mathcal{H}_p(X)$ and $\mathcal{H}^m_p(X)$, respectively.

\begin{theorem} \label{new}
Let $1 \leq p \leq \infty$, $m \in \mathbb{N}$,  $X$ an infinite dimensional Banach space $X$,
and $b = (b_n)$ a multiplicative scalar sequence.
\begin{enumerate}
\item If $X$ has trivial cotype, then
\begin{gather*}
b \in \mult\, \mathcal{H}_p(X) \Leftrightarrow (b_{p_{k}})_{k} \in \ell_1  \cap    B_{c_0} \Leftrightarrow b \in \ell_{1} \\
b \in  \mult\, \mathcal{H}^m_p(X)  \Leftrightarrow (b_{p_{k}})_{k} \in \ell_{1} \,.
\end{gather*}
\item  If $X$ has hypercontractive homogeneous cotype $\ct(X)< \infty$, then
\begin{gather*}
b \in \mult\, \mathcal{H}_p(X) \Leftrightarrow (b_{p_{k}})_{k} \in \ell_{\ct(X)'} \cap     B_{c_0}\Leftrightarrow b \in  \ell_{\ct(X)'}  \\
b \in  \mult\, \mathcal{H}^m_p(X)  \Leftrightarrow (b_{p_{k}})_{k} \in \ell_{\ct(X)'} \,.
\end{gather*}
\end{enumerate}
\end{theorem}

If $X$ has nontrivial cotype but does not satisfies the assumptions of (2), multiplicative multipliers are not completely characterized but we can use \eqref{no-assump} to obtain some information about them.

\medskip
To see an example, we use again the results in Section~\ref{subsec_examples}.

\begin{corollary} \label{nice2} Let $1 \leq p \leq \infty$, $m \in \mathbb{N}$,
and $b = (b_n)$ a multiplicative scalar sequence.
\begin{enumerate}
\item  If\, $1 \leq r \leq \infty$ and $X$ is a $\mathcal L_r$-space,
then
\begin{gather*}
b \in \mult\, \mathcal{H}_p(X) \Leftrightarrow (b_{p_{k}})_{k} \in \ell_{\min \{2,r' \}}  \cap    B_{c_0} \Leftrightarrow b \in \ell_{\min \{2,r' \}}\\
b \in  \mult\, \mathcal{H}^m_p(X)  \Leftrightarrow (b_{p_{k}})_{k} \in \ell_{\min \{2,r' \}} \,.
\end{gather*}

\item If \,$2 \leq r \leq \infty$,
then
\begin{gather*}
b \in \mult\, \mathcal{H}_p(\mathcal{S}_r) \Leftrightarrow (b_{p_{k}})_{k} \in \ell_{r'}  \cap    B_{c_0} \Leftrightarrow b \in \ell_{r'}  \\
b \in  \mult\, \mathcal{H}^m_p(\mathcal{S}_r)  \Leftrightarrow (b_{p_{k}})_{k} \in \ell_{r'} \,.
\end{gather*}
\end{enumerate}
\end{corollary}

\appendix

\section{Cotype with respect to index sets}

Throughout this note we have considered different kinds of \emph{cotypes}: the classical (linear) cotype, homogeneous cotype, hypercontractive homogeneous cotype, Fourier cotype and analytic cotype. We end this note introducing a general setting in which all these concepts can be framed.

Let $\Lambda \subseteq \mathbb{N}_{0}^{(\mathbb N)}$ be an indexing set. We say that the Banach space $X$ has $\Lambda$-cotype $q$ if there exists a constant $C>0$ such that for every finite family $(x_{\alpha})_{\alpha \in\Lambda} \subset X$ (i.e., a family with only finite non-zero elements) we have
 \begin{equation}\label{Lambda-cotype}
 \bigg( \sum_{\alpha \in \Lambda} \Vert x_{\alpha} \Vert^{q} \bigg)^{1/q} \leq C \bigg( \int_{\mathbb{T}^{\infty}} \Big\Vert  \sum_{\alpha \in \Lambda}  x_{\alpha} z^{\alpha} \Big\Vert^{q'} dz \bigg)^{1/q'}.
\end{equation}
 We denote by $C_{q,\Lambda}(X)$ the best constant $C$ satisfying the previous inequality.

\smallskip
The usual notion of cotype turns out to be a particular case of this concept, in the sense that it corresponds to an appropriate choice of the set of multi indices $\Lambda$.  If we take
\[
\Lambda_1=\{\alpha\in \mathbb{N}_{0}^{(\mathbb N)} : |\alpha|=1\} \,,
\]
Then \eqref{Lambda-cotype} with $\Lambda_{1}$ is, through Kahane's inequality, equivalent to   \eqref{cotype}. In other words, $\Lambda_{1}$-cotype is just cotype. \\

The concept of {$m$-homogeneous cotype} can also be seen as a cotype with respect to an indexing set. If we take
\[
\Lambda_m=\{\alpha\in \mathbb{N}_{0}^{(\mathbb N)} : |\alpha|=m\}\,,
\]
and use Proposition~\ref{pol-kahane} (the polynomial Kahane's inequality) then $m$-homogeneous cotype $q$ is  $\Lambda_m$-cotype $q$.
We can rephrase \eqref{homog-cotype} and the subsequent comments: $X$ has $\Lambda_1$-cotype if and only if $X$ has $\Lambda_m$-cotype for some (or for all) $m$ and
\[
C_{q,\Lambda_1}(X)\le C_{q,\Lambda_m}(X)\le
\frac{m^m}{m!} (m!)^{1/q'} K^m
\,\sqrt{\frac{q'}{2}}^m\,C_{q,\Lambda_1}(X)^{m} \,.
\]
% {\color{red} Don't we get another $\sqrt{\frac{q'}{2}}^m$ here, because of our polynomial Kahane inequality?}

Also, hypercontractive homogeneous cotype $q$ means $\Lambda_m$-cotype for all $m$ together with the control of the constants: $C_{q,\Lambda_m}(X)\le C^m$. Hence our conjecture reads:
\[
C_{q,\Lambda_1}(X)\le C_{q,\Lambda_m}(X)\le
\lambda^m \,C_{q,\Lambda_1}(X)^{m} \,
\]
for some universal $\lambda>0$.

For  Fourier cotype, let us  identify $\mathbb N$ as a  subset of $ \mathbb{N}_{0}^{\mathbb N}$ in the natural way
\[
\mathbb N \sim \{\alpha\in \mathbb{N}_{0}^{\mathbb N} : \alpha_k=0 \text{ for }k\ge 2\} \,.
\]
Fourier cotype is $\mathbb N$-cotype and analytic cotype (the inequality in Proposition~\ref{fourier-implies-pol}) is $\mathbb{N}_{0}^{(\mathbb N)}$-cotype.
Finally, note that Proposition~\ref{fourier-implies-pol} states that $\mathbb N$-cotype $q$ is equivalent to $\mathbb{N}_{0}^{\mathbb N}$-cotype~$q$.

\end{document}